\documentclass{amsart}
\title[Variational Hodge conjecture]{Variational Hodge conjecture for complete intersections on hypersurfaces in projective space}
\author[R. Kloosterman]{Remke Kloosterman}
\email{klooster@math.unipd.it}
\address{Universit\`a degli Studi di Padova,
Dipartimento di Matematica ``Tullio Leci-Civita'',
Via Trieste 63,
35121 Padova, Italy}

\theoremstyle{plain}
\newtheorem{theorem}{Theorem}[section]

\newtheorem{lemma}[theorem]{Lemma}
\newtheorem{proposition}[theorem]{Proposition}

\theoremstyle{definition}
\newtheorem{definition}[theorem]{Definition}
\newtheorem{notation}[theorem]{Notation}

\theoremstyle{remark}
\newtheorem{remark} [theorem]{Remark}

\newtheorem{example}[theorem]{Example}

\DeclareMathOperator{\prim}{prim}
\DeclareMathOperator{\NL}{NL}
\DeclareMathOperator{\Ima}{Im}
\DeclareMathOperator{\Ext}{Ext}

\begin{document}

\begin{abstract}
In this paper we give a new and simplified proof of the variational Hodge conjecture for complete intersection cycles on a hypersurface in  projective space, avoiding the use of the semiregularity map.
\end{abstract}

\thanks{The author would like to thank Johan Aise de Jong, Hossein Movasati, Carlos Simpson, Orsola Tommasi and the referee for several remarks on a previous version of this paper.}

\maketitle
\section{Introduction}

The variational Hodge conjecture (see \cite{GroVarHod}) predicts very roughly that for a smooth $n$-dimensional projective complex variety $Y$ and an $m$-dimensional subvariety $Z\subset Y$ one can find $m$-dimensional subvarieties $Z_1,\dots Z_t$ of $Y$ and integers $n_1,\dots,n_t$ such that $[Z]=\sum n_i [Z_i] $ in $H^{2n-2m}(Y,\mathbf{Z})$ and such that the locus $\NL([Z])$ of deformations of $Y$ where the class $[Z] \in H^{2n-2m}(Y,\mathbf{Z})$ remains of type $(n-m,n-m)$ equals  the locus of deformations of $Y$ that can be lifted to  deformations of the pairs $(Y,Z_i)$ for $i=1,\dots,m$. For a precise definition of the Hodge locus see Section~\ref{secGor}.

The variational Hodge conjecture is known to hold for semiregular subvarieties $Z$ (see e.g., \cite{BF}), but it seems quite hard to check whether a given subvariety $Z$ is semiregular.  
%Low degree subvarieties turn out to be  semiregular by \cite{DK}, but the bound on the degree in this paper does not seem to have many applications.

More results are known if $Y$ is a smooth hypersurface in $\mathbf{P}^{n+1}$ of degree $e$. Then the variational Hodge conjecture holds trivially true for all $m$ different from $n/2$. Suppose now that $n$ is even and $m$ equals $n/2$. 
In this paper we consider the case where $Z$ is a complete intersection in $\mathbf{P}^{2m+1}$.
In this case $Z$ is semiregular in $Y$ by \cite[Theorem 3.1]{SteHod} and therefore the variational Hodge conjecture holds for $[Z]$.
Otwinowska \cite{OtwBig} proved the variational Hodge conjecture for $m$-dimensional subvarieties of small degree contained in $Y$.
However, Otwinowska's bounds on the degree of the subvarieties are insufficient to include all complete intersections cycles (cycles of the form $[Z]$ with $Z$ a complete intersection in $\mathbf{P}^{n+1}$). Therefore this case is not completely covered by Otwinowska.

%In this case  

%Then the defining polynomial of $Y$ is a combination over $\mathbf{C}[x_0,\dots,x_{n+1}]$ of  the polynomials defining $Z$, hence $\deg(Z)\leq (e-1)^{(n+1)/2}$.
%However, Otwinowska's bounds on the degree of the subvarieties are insufficient to include all complete intersections cycles (cycles of the form $[Z]$ with $Z$ a complete intersection in $\mathbf{P}^{n+1}$). Therefore this case is not completely covered by Otwinowska.

In \cite{DanVHC}, Dan reproved the variational Hodge conjecture for complete intersection cycles contained in smooth projective hypersurfaces. Recently,  Movasati  and Villaflor-Loyola \cite{MV}  presented a different  strategy to check the variational Hodge conjecture and applied this strategy for certain classes of complete intersection cycles. However, for $n>6$ or $n=2,4$ and the degree of $Y$ is large, they only consider the case  where $Z$ is a linear subspace.

In this paper, we present a different proof for the variational Hodge conjecture for the case of complete intersection cycles contained in hypersurfaces in $\mathbf{P}^{2m+1}$.   This proof follows largely Dan's proof. Both proofs focus on giving an upper bound for the dimension of the tangent space of $\NL([Z]_{\prim})$. For this we use a simple combinatorial argument, where Dan uses a more involved description in terms of commutative algebra.  Our approach seems to be less involved, but Dan's approach seems to have more potential to be generalized.

More precisely, we prove:
\begin{theorem} Let $Y\subset \mathbf{P}^{2k+1}$ be a smooth hypersurface of degree $e\geq 3$, containing a $k$-di\-men\-si\-on\-al subvariety $Z$ which is a complete intersection in $\mathbf{P}^{2k+1}$ of multidegree $(d_1,\dots,d_{k+1})$, such that $d_i<e$ for $i=1,\dots,k+1$.

Then the Hodge locus $\NL([Z]_{\prim})$ is smooth at $Y$ and 
is contained in the locus in $\mathbf{C}[x_0,\dots,x_{2k+1}]_e$ of smooth  hypersurfaces of degree $e$ containing a complete intersection of multidegree $(d_1,\dots,d_{k+1})$.% coincides with the  at a very general point .% with $\gamma=[Z]_{\prim}\in H^{2k}(Y,\mathbf{Q})_{\prim}$.
\end{theorem}

%By Lefschetz hyperplane theorem we know that any subvariety $Z$ of $X$ which is not cohomological equivalent to a multiple of $H^{\codim_X(Z)}$ with $H$ the hyperplane section is of dimension $k$.
Note that the smoothness assumption implies that $d_i\leq e$ for each $i$. If $d_i=e$ then one can take $F$ as one of the generators of $I(Z)$. In that case ${Z}_{\prim}=0$ and $\NL([Z]_{\prim})$ is the locus of smooth hypersurfaces of degree $e$. If $e=2$ then the only possibility for the $d_i$ is to equal 1. This is the case of $2k$-dimensional quadrics. Such a quadric contains a $k$-dimensional linear space. Hence both the Hodge locus as the locus of quadrics containing a smooth $k$-dimensional linear space coincide with the locus of smooth quadrics.

To prove the theorem we note first that if $Y$ contains a complete intersection  $Z$ of multidegree $(d_1,\dots,d_{k+1})$ then it contains a residual complete intersection $\tilde{Z}$ of multidegree $(d_1,\dots,d_{i-1},e-d_i,d_{i+1},\dots, d_{k+1})$. Deformations of the pair $(Y,Z)$  correspond to deformations of the pair $(Y,\tilde{Z})$. 
 The corresponding cycle classes in primitive cohomology differ only by a sign, hence the corresponding Hodge loci  coincide (see Section~\ref{secGor}.) So without loss of generality we may assume  $d_i\leq e/2$ for all $i$.

The first part of the proof of the main theorem holds in much more generality. Let $(X,\mathcal{O}(1))$ be a projective variety of dimension $2k+1$, $S$ be its projective coordinate ring, $Y\in |\mathcal{O}(e)|$ be a smooth hypersurface and $Z$ be a complete intersection of dimension $k$ and multidegree $(d_1,\dots,d_{k+1})$, contained in $Y$. 
The first step of the proof is to give an expression for the dimension  of the so-called ``flag Hilbert scheme" parametrizing flags $(Y,Z)$, consisting of a degree $e$ hypersurface containing a complete intersection $Z$ of multidegree $(d_1,\dots,d_{k+1})$.  We use this to calculate the dimension of the locus $L$ in $S_e$  corresponding to hypersurfaces of degree $e$ containing a complete intersection of multidegree $(d_1,\dots,d_{k+1})$. (See Proposition~\ref{propHSdim}.)

Suppose now that $Z$ is given by $P_1=\dots=P_{k+1}=0$ and $Y$ is given by $\sum_{i=1}^{k+1} P_iQ_i=0$.
Consider the preimage of the Hodge locus $\NL([Z]_{\prim})\subset \mathbf{P}(S_e)$ in $S_e$. 
We use  the above dimension calculation to show that    $\NL([Z]_{\prim})$ is smooth at $Y$ and of  dimension equal to $\dim L$ if and only if the tangent space to $\NL([Z]_{\prim})$ at $Y$ inside $S_e$ equals the degree $e$ part of the ideal \[(P_1,\dots,P_{k+1},Q_1,\dots,Q_{k+1}).\]

The final part of the proof is to show that the latter condition holds in the case $X=\mathbf{P}^{2k+1}$. For this we construct a  certain Artinian Gorenstein algebra $\mathbf{C}[x_0,\dots,x_{2k+1}]/I$ such that $I_e$ is the tangent space to the Hodge locus. We show that in the case of a complete intersection cycle the ideal $I$ is generated by $P_1,\dots,P_{k+1},Q_1,\dots,Q_{k+1}$. 
This Artinian Gorenstein algebra plays a crucial role in many papers on the Noether-Lefschetz loci and Hodge loci of hypersurfaces in $\mathbf{P}^{2k+1}$.

The structure of the paper is as follows:
In Section~\ref{secCIIdeals} we discuss the Hilbert function of a complete intersection ideal and give an expression for it. In Section~\ref{secHS} we study the flag Hilbert scheme of hypersurfaces of dimension $2k$ and degree $e$  containing a complete intersection of dimension $k$. We obtain an expression for the dimension of this Hilbert scheme in terms of the Hilbert function of certain complete intersection ideals. In Section~\ref{secGor} we study the tangent space to the Hodge locus by constructing an Artinian Gorenstein algebra and obtain an expression for the dimension of the tangent space. From this we deduce the variational Hodge conjecture for complete intersection cycles on hypersurfaces in $\mathbf{P}^{2k+1}$.

\section{Complete intersection ideals}\label{secCIIdeals}
Let $X$ be  a projective variety and let $\mathcal{O}(1)$ be a very ample line bundle on $X$. Consider the embedding of $X$ into $\mathbf{P}^N$, with $N=h^0(\mathcal{O}(1))-1$. Let $I(X)$ be the homogeneous ideal of $X$ and let $S(X)=\mathbf{C}[X_0,\dots,X_N]/I(X)$ be the projective coordinate ring of $X$, which is a graded ring. % (with the grading depending on the choice of $\mathcal{O}(1)$). 
Let $\chi: \mathbf{Z} \to \mathbf{Z}$ be the Hilbert function of $X$ with respect to $\mathcal{O}(1)$, defined by
\[ \chi(m)=\dim S(X)_m.\]
For any homogeneous ideal $I$ of $S(X)$ define the Hilbert function $h_I: \mathbf{Z}_{\geq 0} \to \mathbf{Z}_{\geq 0}$ of $I$ by
\[ h_I(m)=\dim (S(X)/I)_m.\]
Fix a complete intersection ideal  $I=(P_1,\dots,P_t)$, i.e., $P_1,\dots,P_t$ form a regular sequence of homogeneous elements of $S$. Let $d_i:=\deg(P_i)$. Assume that the $P_i$ are numbered such that  $1\leq d_1\leq d_2\leq\dots\leq d_t$. 

Consider now the Koszul complex associated with the map $\oplus S(-d_i) \to S$ given by $(P_1,\dots,P_t)$. 
\begin{equation}\label{eqnHilb} 0 \to \oplus_{j=1}^{k_t} S(-a_{tj}) %\to \oplus_{j=1}^{k_{t-1}} S(-a_{(t-1)j}) 
\to \dots \to \oplus_{j=1}^{k_1} S(-a_{1j}) \to S \to S/I \to0\end{equation}
This complex is exact since $(P_1,\dots,P_t)$ is a regular sequence. Therefore we can use this complex to express the Hilbert function of the complete intersection $Z=Z(P_1,\dots,P_t)$ in terms of the Hilbert function of $X$.

First, the $a_{ij}$ can be determined as follows:
 For fixed $i$ we find that the multiset $\{a_{ij}\}$ coincides with the multiset $\{d_{k_1}+d_{k_2}+\dots+d_{k_i}\}$
with $0<k_1<k_2<\dots<k_i\leq t$. Equivalently, the $a_{ij}$ are the degrees of the products of the elements of subsets of $i$ distinct elements in $\{P_1,\dots,P_t\}$.

By the above mentioned exact sequence we now find that for $m\geq 0$.
 \begin{equation}\label{eqnDim} h_I(m)=\chi(m)+\sum_{i=1}^t (-1)^i\sum_{j=1}^{k_i} \chi(m-a_{ij}).\end{equation}
For $m'<0$ we have $\chi(m')=0$ . Hence in the above sum we are allowed to leave out any $a_{ij}$ which is at least $m+1$.

In the sequel, we want to concentrate on a special case and consider the difference between the Hilbert functions of two very similar ideals. For this we assume that $X$ has odd dimension, i.e., $\dim X=2k+1$.

\begin{notation}
Fix  $k+1$ homogeneous polynomials $P_1,\dots,P_{k+1}$, which form a regular sequence %(i.e, they define  a complete intersection) 
such that $0<d_1\leq d_2\leq\dots\leq d_{k+1}$. Let $I'$ be the ideal generated by $P_1,\dots,P_{k+1}$.
Fix a further integer $e\geq 2d_{k+1}$, and $k+1$ homogeneous polynomials $Q_1,\dots, Q_{k+1}$ such that $\deg(Q_i)=e-\deg(P_i)$ and $P_1,\dots,P_{k+1},Q_{k+1},\dots,Q_1$ forms a regular sequence in $S$. Let $I=(P_1,\dots,P_{k+1},Q_1,\dots,Q_{k+1})$.
\end{notation}

The key ingredient of our proof of the variational Hodge conjecture is a formula for  $h_I(e)-h_{I'}(e)$. To this aim, we introduce the following notation.

\begin{notation}\label{notDegs} If $e$ is odd let $a=a(I')=0$, if $e$ is even let $a=a(I')=\# \{i \mid d_i=e/2\}$.
Set $c=c(I')=\frac{1}{2}a(a-1)$. 

For any two integers $i,j\in \{1,\dots, k+1\}$, let us define
\[ A(i;j) =\left \{(k_1,\dots,k_i)\mid 0<k_1<k_2<\dots<k_i\leq k+1; d_j\geq \sum_{m=1}^i d_{k_m} \right\}\]
For consistency reasons we also introduce the set $A(0;j)=\{0\}$ and set $P_0=1$. 
\end{notation}

The numbers $a,c$ and the set $A(i;j)$ depend only on the $d_i$, but to simplify notation we write $a(I'), c(I')$ etc. when necessary.

\begin{lemma}\label{lemHilbDiff}The difference  $h_{I}(e)-h_{I'}(e)$ equals
\[ c(I')+\sum_{i=0}^{k+1}\sum_{j=1}^{k+1} \sum_{(k_1,\dots,k_i)\in A(i;j)}   (-1)^{i+1}\chi \left(d_j -d_{k_1}-d_{k_2}-\dots-d_{k_i}\right).\]
\end{lemma}
\begin{proof}
Consider the Koszul complex associated with $(P_1,\dots,P_{k+1})$
\begin{equation} 0 \to \oplus_{j=1}^{k'_{t'}} S(-a'_{t'j}) %\to \oplus_{j=1}^{k'_{t'-1}} S(-a'_{(t'-1)j}) 
\to \dots \to \oplus_{j=1}^{k_1'} S(-a'_{1j}) \to S \to S/I' \to0\end{equation}
and the Koszul complex associated with $(P_1,\dots,P_{k+1},Q_1,\dots Q_{k+1})$
\begin{equation} 0 \to \oplus_{j=1}^{k_{t}} S(-a_{tj}) %\to \oplus_{j=1}^{k_{t-1}} S(-a_{(t-1)j}) 
\to \dots \to \oplus_{j=1}^{k_1} S(-a_{1j}) \to S \to S/I \to0\end{equation}

Let $M_i:= \oplus_{j=1}^{k_i} S(-a_{ij})$ and $M'_i:= \oplus_{j=1}^{k'_i} S(-a'_{ij})$, where we set $M'_i=0$ for $t'<i\leq t$. The $a'_{ij}$ are the total degrees of subsets of $i$ elements from $\{P_1,\dots,P_{k+1}\}$, the 
 $a_{ij}$ are the total degree of subset of $i$ elements from $\{P_1,\dots,P_{k+1},Q_1,\dots,Q_{k+1}\}$. In particular, $M'_i$ is a natural submodule of $M_i$. Moreover, $M_i/M'_i$ is free and is isomorphic with $\oplus_{j=1}^{k'_i-k_i} S(-b_{ij})$ where the $b_{ij}$ are the total degree of subsets of length $i$ from  $\{P_1,\dots,P_{k+1},Q_1,\dots,Q_{k+1}\}$ containing at least one $Q_m$. 
 
 From (\ref{eqnDim}) it follows that the difference $h_I(e)-h_{I'}(e)$ equals
 \[ \chi(e)+\sum_{i=1}^t (-1)^i \dim (M_i)_e -\left(\chi(e)+\sum_{i=1}^t (-1)^i \dim (M'_i)_e\right),\] which in turn equals $\sum_{i=1}^t (-1)^{i} \dim (M_i/M'_i)_e$.
 Recall that $(M_i/M'_i)_e=\oplus_{j=1}^{k'_i-k_i} S(-b_{ij})_e$. Hence
 \[ \dim (M_i/M'_i)_e=\sum_{j=1}^{k'_i-k_i} \chi(e-b_{ij}).\]
 If $b_{ij}>e$ then the contribution to the summand is zero. Since $\deg Q_m\geq e/2$ and we need to use at least one of the $Q_m$ to obtain a $b_{ij}$, we have two ways to obtain $b_{ij}\leq e$. Firstly, $b_{ij}$ can correspond to the degree of a subset  of the form$P_{k_1},
 \dots, P_{k_{i-1}}, Q_{j}$ such that the total degree $d_{k_1}+\dots d_{k_{i-1}}+e-d_{j}$ is at most $e$, i.e.,  $d_{k_1}+\dots d_{k_{i-1}}\leq d_{j}$, or $j=2$ and $b_{ij}$ corresponds to $Q_{k_1},Q_{k_2}$ with $d_{k_1}=d_{k_2}=e/2$. In the latter case the total degree equals $e$.
 
 In particular, for $i\neq 2$ we have that 
 \[ \dim (M_i/M'_i)_e= \sum_{(k_1,\dots,k_{i-1})\in A(i-1;j)}  \chi \left(d_j -d_{k_1}-d_{k_2}-\dots-d_{k_{i-1}}\right)\]
and for $i=2$
 \[ \dim (M_i/M'_i)_e=c(I')+ \sum_{(k_1,\dots,k_{i-1})\in A(i-1;j)}  \chi \left(d_j -d_{k_1}-d_{k_2}-\dots-d_{k_{i-1}}\right).\]
 Combining  this with $ h_{I}(e)-h_{I'}(e)=\sum_{i=1}^ t (-1)^{i} \dim (M'_i/M_i)_e$ yields
 \[ c(I')+\sum_{i=0}^{k+1}\sum_{j=1}^{k+1} \sum_{(k_1,\dots,k_i)\in A(i;j)}   (-1)^{i+1}\chi \left(d_j -d_{k_1}-d_{k_2}-\dots-d_{k_i}\right).\]
\end{proof}

\section{Hilbert schemes of complete intersections}\label{secHS}
Fix a positive integer $k$  and let $(X,\mathcal{O}(1))$ be a projective variety of dimension $2k+1$ with an ample line bundle $\mathcal{O}(1)$. For the rest of this section fix an integer $e>0$ and fix positive integers $1\leq d_1\leq d_2 \leq \dots \leq d_{k+1}\leq e/2$. We want to calculate the dimension of the sublocus of $|\mathcal{O}(e)|$  consisting of hypersurfaces containing a complete intersection of multidegree $(d_1,\dots,d_{k+1})$.

For $t\leq k+1$ we denote with $H(d_1,\dots,d_t)$ the  Hilbert scheme of complete intersections $Z$ of multidegree $(d_1,\dots,d_t)$ contained in $X$. As in the previous section we can associate with $i,j \in \{1,\dots, t\}$ a  set $A(i;j)$ (see Notation~\ref{notDegs}).

\begin{lemma}\label{lemHilbSchB} If $ H(d_1,\dots,d_t)$ is not empty then it is irreducible and its  dimension equals
\[
 \sum_{j=1}^t \sum_{i=0}^t (-1)^i \sum_{(k_1,\dots,k_i)\in A(i;j)}  \chi (d_j-d_{k_1}-\dots -d_{k_i}).\]
 \end{lemma}
\begin{proof}
Recall that we assumed that $d_1\leq d_2\leq \dots \leq d_t$. 

We prove the claim by induction on the number of distinct values in $d_1,\dots, d_t$.

If there is one value, i.e., if $d_t=d_1$ then $H(d_1,\dots,d_t)$ is an open subset of the Grassmannian of $t$-dimensional subspaces of $S_{d_t}$, which  is irreducible of dimension
$t(\chi (d_t)-t).$
In this case $A(0;j)=\{0\}$ and it contributes $\chi(d_j)=\chi(d_t)$ to the sum, $A(1;j)=\{d_1,\dots,d_t\}$ and each element contributes $-1$ to the sum. Hence we find that the right hand side equals
\[ t(\chi(d_t)-t)\]
and we are done in this case.

If $d_t\neq d_1$ then let $r$ be the largest integer such that $d_r<d_t$. We have a natural map 
\[ H(d_1,\dots,d_t)\to H(d_1,\dots,d_{r})\]
mapping the complete intersection $Z:=V(P_1,\dots,P_t)$ to the complete intersection $Z':=V(P_1,\dots,P_r)$.

The fiber over $Z'$ is an open subset in the Grassmannian of dimension $t-r$ subspaces in $(S/I(Z'))_{d_t}$. Hence the fiber dimension equals
\[(t-r) (h_{I(Z')}(d_t)-(t-r))\]
which equals
\[ \sum_{j=r+1}^t\left( h_{I(Z')}(d_j)-(t-r)\right).\]
To compute $h_{I(Z')}(d_j)$ for $j=r+1,\dots,t$ we apply formula (\ref{eqnHilb}) and consider the resolution of the ideal  $I(Z')$.
The generators of $I(Z')$ correspond to the elements of $A(1;j)$ of degree strictly less than $d_j=d_t$, the $i-1$-th syzygies of degree at most $d_j$ correspond to the elements of $A(i;j)$ with $2\leq i \leq r$. 
Note that $A(i;j)$ is empty for $i>r$.

Hence $h_{I(Z')}(d_j)$ equals
\[ \chi(d_j)+\sum_{i=1}^r (-1)^i \sum_{(k_1,\dots,k_i)\in A(i;j)}  \chi (d_j-d_{k_1}-\dots -d_{k_i}) + (t-r)\]
 (The contribution $t-r$ is a correction for the fact that we sum also over the elements in $A(1;j)$ of degree $d_t$.)
 
By induction we now obtain that 
\[ \dim H(d_1,\dots,d_t) = \sum_{j=1}^t \sum_{i=0}^t (-1)^i \sum_{(k_1,\dots,k_i)\in A(i;j)}  \chi (d_j-d_{k_1}-\dots -d_{k_i}).\]

The irreducibility claim now follows by induction: A flag Hilbert scheme para\-met\-rizing flags of subvarieties in a projective variety is itself a projective scheme, see \cite[Theorem 4.5.1]{Ser}. This implies that the projection morphism $H(d_1,\dots,d_t)\to H(d_1,\dots,d_r)$ is proper. Moreover, we showed that the fiber dimension is constant and every fiber is irreducible. It is a straightforward exercise to show that if $f:S\to T$ is a proper morphism onto an irreducible scheme $T$, with constant fiberdimension and irreducible fibers then $S$ is irreducible.
The scheme $ H(d_1,\dots,d_r)$ is irreducible by the induction hypothesis, hence $H(d_1,\dots,d_t)$ is irreducible.
\end{proof}

Denote now by $H(d_1,\dots,d_{k+1};e)$ the flag Hilbert scheme of pairs $(Y,Z)$ consisting of complete intersections $Z$ of multidegree $(d_1,\dots,d_{k+1})$ contained in a hypersurface $Y$ in $X$ of degree $e$. The dimension of the fiber of $H(d_1,\dots,d_{k+1};e)\to H(d_1,\dots, d_{k+1})$ over $Z$ equals $\mathbf{P}(I_e(Z))$ and hence has dimension $\chi(e)-h_{I(Z)}(e)-1$. Hence $\dim H(d_1,\dots,d_{k+1};e)$ can be easily obtained from the previous lemma.

Let $H(e)=\mathbf{P}(S_e)$ be the Hilbert scheme of hypersurfaces of degree $e$ in $X$. 

\begin{proposition}\label{propHSdim} Suppose that $\dim X$ is odd. Let $k=\frac{\dim X-1}{2}$. 

The locus  $L(d_1,\dots,d_{k+1};e)$ of smooth hypersurfaces in $H(e)$ containing a complete intersection of multidegree $(d_1,\dots,d_{k+1})$ has codimension $h_I(e)$, where $I$ is a complete intersection ideal of multidegree $(d_1,d_2,\dots, d_{k+1},e-d_{k+1},\dots,e-d_1)$.
\end{proposition}
\begin{proof}
We start by determining the dimension of the general fiber of \[\pi: H(d_1,\dots,d_{k+1};e) \to \Ima(\pi)\subset H(e),\] where $\Ima(\pi)=L(d_1,\dots,d_{k+1};e)$ by definition.

The tangent space to $\pi^{-1}(Y)$ at $(Y,Z)$  is contained in $H^0(N_{Z/Y})$. The latter vector space is the kernel of $H^0(N_{Z/X})\to H^0(N_{Y/X}|_{Z})$. 

If $Z$ is given by $P_1=\dots= P_{k+1}=0$ then $Y$ is given by $\sum P_iQ_i=0$ for some proper choices of $Q_1,\dots Q_{k+1}$.
The following argument seems to be well-known but we could not find a proper reference.
Let $\mathcal{I}$ be the ideal sheaf of $Z$ in $X$. Then the final part of the Koszul resolution of $Z$ equals
\[ \oplus \mathcal{O}_X(-d_i-d_j) \to \oplus \mathcal{O}_X(-d_i) \to \mathcal{I} \to 0\]
Using that $\mathcal{I}/\mathcal{I}^2\cong \mathcal{I} \otimes \mathcal{O}_Z$, we find that after tensoring the above exact sequence with $\mathcal{O}_Z$ that 
\[ \oplus \mathcal{O}_Z(-d_i-d_j) \to \oplus \mathcal{O}_Z(-d_i) \to \mathcal{I}/\mathcal{I}^2 \to 0\]
is exact. However the first map is the zero map on $Z$, hence the normal bundle of $Z/X$ (the dual of $\mathcal{I}/\mathcal{I}^2$) is isomorphic to $\oplus \mathcal{O}_Z(d_i)$.
In particular $H^0(N_{Z/X})$ can be identified with $\oplus_{i=1}^{k+1} S/I(Z)(d_i-e)_e$. Since $N_{Y/X}=\mathcal{O}_Y(e)$ we obtain that $H^0(N_{Y/X}|_{Z})$ can be identified with $(S/I(Z))_e$. The map between $H^0(N_{Z/X})\to H^0(N_{Y/X}|_{Z})$  corresponds to multiplication by $(Q_1,\dots,Q_{k+1})$. 

We claim that  $\Sigma: P_1=\dots=P_{k+1}=Q_{k+1}=\dots =Q_1=0$ is a complete intersection. If not then $\Sigma$ would be a non-empty subscheme of $X$ along which $Y$ would be singular. Since we assumed $Y$ to be smooth, this is not the case.
In particular, $Q_1=\dots=Q_{k+1}=0$ is a complete intersection in $Y$. Hence we can study the kernel of the map $\oplus S/I(Z)(d_i-e) \to S/I(Z)$ by studying the associated Koszul complex.

However, we are only interested in the degree $e$ part. Each of the  $Q_i$ has degree at least $e/2$, hence there are few syzygies of degree at most $e$: Every second syzygy has degree bigger than $e$ and if there are first syzygies of degree at most $e$ then they have degree precisely $e$, and these syzygies correspond to unordered pairs of $\{Q_i,Q_j\}$ such that $\deg Q_i=\deg Q_j=e/2$. Hence there are $c=c(I)$ independent syzygies of this form. So we find that
\[ 0 \to \mathbf{C}^{c} \to \oplus S/I(Z)(d_i-e)_e\to S/I(Z)_e\to 0\]
is exact and $\dim H^0(N_{Z/Y})=c$. This gives an upper bound on the dimension of the tangent space to the fiber and hence an upper bound for the fiber dimension. We show now that the fiber dimension equals the upper bound on the dimension of the tangent space.

If $c=0$ then we find that the fibers are empty or finite and we are done. 

If $c>0$ then  let $a=a(I(Z))$ and let $r=k+1-a(I(Z))$ be the largest integer such that $d_r<e/2$.

Fix now an $a\times a$-antisymmetric matrix $M$ with entries from $\mathbf{C}$. 
For $i=1,\dots,a$ define polynomials $R_i=P_{r+i}+\sum_{j=1}^a M_{ij} Q_{r+j}$.
Then one easily checks that
\[ \sum_{i=1}^r Q_iP_i+\sum_{j=1}^a Q_{r+j} R_j=\sum_{i=1}^{k+1} P_iQ_i\]
holds.
I.e., the complete intersection given by $P_1=P_2=\dots=P_r=R_1=\dots=R_a$ is a complete intersection of the same multidegree as $Z$ and contained in $Y$.
Moreover, two distinct matrices yields distinct complete intersections. Hence the fiber dimension is at least $\frac{1}{2}a(a-1)=c$. Combining this with the previous upper bound, we find that the fiber dimension equals $c$.

Since the dimension of the fiber of $\pi$ equals $c$, it follows from the fiber dimension theorem that the image of $\pi$ (which is precisely $L(d_1,\dots,d_{k+1};e)\subset \mathbf{P}(S_e)$)
has dimension $ H(d_1,\dots,d_{k+1};e)-c$.  By the remark before this Proposition we know that $\dim H(d_1,\dots,d_{k+1};e)=\dim H(d_1,\dots,d_{k+1})+\chi(e)-h_{I(Z)}(e)-1$. Combining this we find 
\[\dim L(d_1,\dots,d_{k+1};e)= \dim H(d_1,\dots,d_{k+1}) + \chi(e)-h_{I(Z)}(e)-1 - c.\]
So the codimension of $ L(d_1,\dots,d_{k+1};e)$ equals
\[ h_{I(Z)}(e)+c-\dim H(d_1,\dots,d_{k+1}).\]
By Lemma~\ref{lemHilbDiff} and Lemma~\ref{lemHilbSchB}  this equals $h_I(e)$.
\end{proof}

\section{Hodge loci and the Artinian Gorenstein algebra associated with a Hodge locus}\label{secGor}

Let $k$ be a  positive integer. Let $S=\mathbf{C}[x_0,\dots,x_{2k+1}]$ be the polynomial algebra in $2k+2$ variables, with its natural grading.
Let $U\subset \mathbf{P}(\mathbf{C}[x_0,\dots,x_{2k+1}]_e)$ be the open set corresponding to smooth hypersurfaces of degree $e$.

Let $X=Z(F)$ be a smooth hypersurface of degree $e$, such that $X$ contains a non-zero Hodge class $\gamma\in H^{k,k}(X,\mathbf{C})_{\prim}\cap H^{2k}(X,\mathbf{Z})_{\prim}$. 
We will now introduce the Hodge locus inside $U$, following \cite[Section 5.3.1]{Voi2}.

Let $B\subset U$ be a simply connected neighborhood (in the analytic topology) of $[F]$. The variation of Hodge structures on $H^{2k}(X,\mathbf{Z})$ over $B$ yields a local system $H$ over $U$ and a decreasing filtration $F$, the Hodge filtration, on the holomorphic vector bundle $\mathcal{H}:=H\otimes \mathcal{O}_B$.

\begin{definition}
The class $\gamma$ induces a natural section $\lambda$ of $H$. For $b\in B$ denote with $\lambda_b$ the value of $\lambda$ in $H^{n}(X_b,\mathbf{Z})_{\prim}$. We define $\NL(\gamma)\subset B$ to be the locus $\{b\in B \mid \lambda_b \in F^{k} \mathcal{H}_b\}$. 
\end{definition}
The locus $\NL(\gamma)$ parametrizes the locus of deformations where the rational class $\gamma$ remains of type $(k,k)$. It is known that
$\NL(\gamma)$ is a complex analytic subset with a natural scheme structure. We can glue these local parts to an analytic scheme with possibly countable many irreducible components.

\begin{definition}
The Hodge locus is the union of all $\NL(\gamma)$ in $U$, where we vary both $X$ and $\gamma$.
\end{definition}
From the work of Cattano, Deligne and Kaplan \cite{DelKap} it follows that the Hodge locus is a countable union of algebraic subsets. In particular, if $L$ is an irreducible component of the Hodge loucs, then in a neighborhood $U$ of a very general point we have $L\cap U=\NL(\gamma)$ for some $\gamma$.

\begin{lemma}\label{lemRedDeg} Let $X=Z(F)$ be a smooth hypersurface of degree $e$. Suppose $Z\subset X$ is a complete intersection of multidegree $(d_1,\dots,d_{k+1})$. Then for every $i\in \{1,\dots, k+1\}$ we have that $X$ contains a complete intersection $Z'$ of multidegree $(d_1,\dots,d_{i-1},e-d_i,d_{i+1},\dots,d_{k+1})$. Moreover $\NL([Z]_{\prim})=\NL([Z']_{\prim})$  in a neighborhood of $[F]$.
\end{lemma}

\begin{proof}
Suppose $Z=V(P_1,\dots,P_{k+1})$, where $P_i$ is homogeneous of degree $d_i$. From the fact that $Z$ is contained in $X$ it follows that $F=\sum_{i=1}^{k+1} P_iQ_i$ for some $Q_i$ homogeneous of degree $e-d_i$. 
In particular, $X$ contains the complete intersection $Z'=V(P_1,\dots,P_{i-1},Q_i,P_{i+1},\dots,P_{k+1})$. The latter is of multidegree  $(d_1,\dots,d_{i-1},e-d_i,d_{i+1},\dots,d_{k+1})$. 

Consider now $Z''=V(P_1,\dots,P_{i-1},P_{i+1},\dots,P_{k+1})$, the schemetheoretic union of $Z$ and $Z'$.  Since $F\equiv P_iQ_i \bmod I(Z'')$ we have  $[Z]+[Z']=[Z'']$.
At the same time $[Z'']$ equals  $\frac{\prod_{j=1}^{k+1} d_j}{d_i} [H]^{k}$, where $[H]$ is the hyperplane class in $H^{2k-2}(X,\mathbf{Z})$. From $([H]^k)_{\prim}=0$ in $H^{2k}(X,\mathbf{Q})_{\prim}$ it follows that $[Z]_{\prim}=-[Z']_{\prim}$ in  $H^{2k}(X,\mathbf{Z})_{\prim}$. Since $\NL(\gamma)$ only depends on the ray of $\gamma$ in $H^{2k}(X,\mathbf{Z})_{\prim}$ we obtain  that $\NL([Z]_{\prim})=\NL([Z']_{\prim})$.
\end{proof}

\begin{remark} Hence to prove the Variational Hodge conjecture for complete intersections it suffices to prove this under the extra assumption $e_i\leq d/2$ for all $d$.

\end{remark}

A central object in the study of Hodge loci is the Artinian Gorenstein algebra associated with it, which we now want to introduce.

Suppose $R$ is a zero-dimensional local ring. Then Grothendieck duality  yields a nondegenerate pairing
\[ R \times \Ext^0(R,\omega_R) \to \mathbf{C}\]
If $R$ is Gorenstein then by definition $\omega_R \cong R$ and we obtain an isomorphism $\Ext^0(R,\omega_R)\cong R$. In particular we have a nondegenerate pairing on $R$.

Suppose now that we have a homogeneous ideal $I=(P_0,\dots,P_{2k+1})$ in $S$,  such that its \emph{affine} zerolocus $V(P_0,\dots,P_{2k+1})$ is s zerodimensional affine scheme.
Then its affine coordinate ring $S/I$ has finitely many prime ideals.  Since $I$ is homogeneous we find that the only prime ideal is the irrelevant ideal $(x_0,\dots,x_{2k+1})$. Therefore $S/I$ is zero-dimensional local complete intersection ring. In particular $S/I$ is Gorenstein \cite[Corollary 21.19]{EisCA}.
Let $m$ be the maximum degree such that $I_m\neq S_m$ then it turns out $\dim (S/I)_m=1$ and that the pairing coming from Grothendieck duality
\[ S/I\otimes S/I\to \mathbf{C}\]
respects the grading and is (upto a scalar) induced by the multiplication map
\[ (S/I)_k \times (S/I)_{m-k} \to (S/I)_m\]
(see \cite[Section 9]{Huneke}.)

Moreover, since $S/I$ has a unique maximal ideal and is Noetherian, it follows from Hopkin's theorem that $S/I$ is Artinian. Equivalently we can introduce Artinian Gorenstein algebras as follows:

\begin{definition}
A \emph{graded Artinian Gorenstein algebra} over $\mathbf{C}$ is a  graded $\mathbf{C}$-algebra $R$  which has finite dimension as a vector space, such that there is a positive integer $m$ such that $R_i=0$ for $i>m$; $\dim R_m=1$ and for every integer $i$ with $0\leq i \leq m/2$ we have that the multiplication map
\[ R_i \times R_{m-i} \to R_m\]
is a perfect pairing.

We call $m$ the \emph{socle degree} of $R$.
\end{definition}

\begin{example}\label{ExaGorensteinDefPol} 
Let $Y\subset \mathbf{P}^{2k+1}$ be a smooth hypersurface of degree $e$. Let $F\in S_e$ be a defining polynomial for $Y$. Assume that  $Y$ contains a $k$-dimensional complete intersection $Z$ of multidegree $(d_1,\dots,d_{k+1})$, such that  $1\leq d_1\leq d_2\leq d_3\leq \dots\leq d_{k+1}\leq e/2$. Let $P_1,\dots, P_{k+1}$ be generators of the ideal of $Z$. Then $F$ is of the form 
\[ P_1Q_{1}+\dots +P_{k+1}Q_{k+1}.\]
Any common point of $P_1=\dots =P_{k+1}=Q_1=\dots=Q_{k+1}=0$ is a singular point of the affine cone of $Y$ in $\mathbf{A}^{2k+2}$. Since $Y\subset \mathbf{P}^{2k+1}$ is smooth, the only such point is the origin, hence these $2k+2$ polynomials define a scheme-theoretic complete intersection in $\mathbf{A}^{2k+2}$. In particular,
\[S/(P_1,\dots,P_{k+1},Q_1,\dots,Q_{k+1})\]
is an example of an Artinian Gorenstein algebra. Its socle degree equals the sum of the degrees of the generators minus the number of variables, which equals \[ (k+1)e-2k-2.\]
\end{example}

%\textbf{?} One easily checks that a graded algebra as above satisfied $\dim R_0=1$, and that $m$ has a unique maximal ideal $\oplus_{i>0}R_i$. In particular any graded Artinian Grorenstein algebra is an Aritinian ring. Moreover the existence of the pairing easily yields that $R$ is Gorenstein.\textbf{?}

\begin{remark}\label{rmkQuo}
If $S/I$ and $S/J$ are  graded Artinian Gorenstein algebras and $S/J$ is a quotient of $S/I$ then  $J=(I: g)$ for some homgeneous form $g$. The difference in the socle degree equals the degree of $g$. Hence the only quotient of $S/I$ with the same socle degree is $S/I$ itself. See \cite[Theorem 21.6]{EisCA}
\end{remark}
%An example of an Artinian Gorenstein algebra is given by the choice of a homogeneous complete intersection ideal $I\subset S$, such that the zero locus of $I$ defines a zero-dimensional affine scheme. In this case $S/I$ is finite dimensional. Since $I$ is homogeneous and its affine zeroset is zero dimensional we have that the zeroset is supported on the origin and that $R=S/I$ has a unique maximal ideal, in particular $R$ is a local ring. Every local complete intersection ring is Gorenstein. Hence $R$ is an Artinian Gorenstein ring.\textbf{Laatste stap}

\begin{lemma} \label{lemRecover} Let $m\geq 1$. Suppose $V\subset S_m$ is a subspace of codimension 1. Suppose $V$ is base point free.  Then there is a unique ideal $I$ such that $S/I$ is Artinian Gorenstein and $I_m=V$.
\end{lemma}

\begin{proof}
Let $I'$ be the ideal generated by $V$.
From \cite[Theorem 2]{GreenF} it follows that $h_{I'}(m+1)=0$. In particular $S/I'$ is finite dimensional. Consider the multiplication map
\[ S_k\times S_{m-k} \to (S/I')_m\cong \mathbf{C}\]
The kernel on the right equals
\[ \{F \in S_k \mid FS_{m-k} \subset V\}\]
So for $k\leq m$ set $I_k=\{ F\in S_k \mid FS_{m-k} \subset V\}$ and $I_k=S_k$ for $k>m$. Then the pairing 
\[ (S/I)_k\times (S/I)_{m-k} \to (S/I)_m\cong \mathbf{C}\]
is perfect by construction. Moreover it is immediate that $I$ is an ideal. Any other ideal $J$ such that $V\subset J_m$ and $S/J$ is Artinian Gorenstein with socle degree $m$ satisfies $I\subset J$. By Remark~\ref{rmkQuo} it follows that  $S/I$ has no nontrivial quotient which is both Gorenstein  and has the same socle degree. Hence we find that $I=J$.
\end{proof}

\begin{proposition}
Let $Y\subset \mathbf{P}^{2k+1}$ be a smooth hypersurface of degree $e$. Let $F\in S_e$ be a defining polynomial for $Y$.
Let $J\subset S$ be the Jacobian ideal of $F$. Then $H^{k,k}(Y)_{\prim}$ can be naturally identified with $(S/J)_{(k+1)e-2k-2}$. Moreover the cup product map $H^{k,k}(Y)_{\prim}\times H^{k,k}(Y)_{\prim} \to H^{2k,2k}(Y)$ can be identified (up to a nonzero constant) with the multiplication map
\[ ( S/J)_{(k+1)e-2k-2} \times (S/J)_{(k+1)e-2k-2} \to (S/J)_{2(k+1)e-4k-4}.\]
\end{proposition}

\begin{proof}
The identification of the primitive part of the middle  cohomology of a smooth hypersurface is the main Theorem of \cite{GriRat}.

The fact that one can identify the cupproduct map by the multiplication map is  the main result of  \cite{CarGri} .
\end{proof}

We can use the above Proposition to associate a Artinian Gorenstein algebra with $\gamma$.

\begin{definition}\label{defGorensteinHL} Suppose that $e\geq 3$.
Let $Y$ be a smooth hypersurface in $\mathbf{P}^{2k+1}$ of degree $e$. Let $\gamma$ be a Hodge class in  $H^{k,k}(Y)_{\prim}\cap H^{2k}(Y,\mathbf{Q})$.

Consider $\gamma ^{\perp}$ in $H^{k,k}(Y)_{\prim}\cong (S/J)_{(k+1)e-2k-2}$, where the orthogonal complement is taken with respect to the cup product. 

Let $W$ be its inverse image in $S_{(k+1)e-2k-2}$. Then $W$ has codimension 1, and is obviously base point free since it contains the degree $(k+1)e-2k-2$ part of the Jacobian ideal, and the Jacobian ideal is base point free in degree $\geq e-1$. (Here we use $e>2$.)

 We can apply the construction of Lemma~\ref{lemRecover} to obtain an ideal $I$ such that $I_{(k+1)e-2k-2}=W$ and $S/I$ is Artinian Gorenstein of socle degree $(k+1)e-2k-2$. We call $S/I$ the \emph{Artinian Gorenstein algebra associated with $(Y,\gamma)$}.
\end{definition}

\begin{lemma}
Let $Y\subset \mathbf{P}^{2k+1}$ be a smooth hypersurface of degree $e$. Let $F\in S_e$ be a defining polynomial for $Y$. Assume that  $Y$ contains a $k$-dimensional complete intersection $Z$ of multidegree $(d_1,\dots,d_{k+1})$, such that  $1\leq d_1\leq d_2\leq d_3\leq \dots\leq d_{k+1}\leq e/2$. Let $P_1,\dots, P_{k+1}$ be generators of the ideal of $Z$. 
Let $Q_1,\dots,Q_{k+1}$ be further forms such that 
\[F= P_1Q_{1}+\dots +P_{k+1}Q_{k+1}.\]
Denote with $\gamma=[Z]_{\prim}\in H^{k}(Y)_{\prim}$. %Then $\gamma \neq 0$.% The Hodge locus of $\gamma$, denoted by $\NL(\gamma)\subset \mathbf{P}(S_e)$, is the locus of hypersurfaces where $\gamma\in H^{2k}(Y,\mathbf{Q})$ remains of type $(k,k)$. %This locus can be defined using deformation theory and has a natural scheme structure (see e.g., the main result of \cite{DelKap}). 

The lift to $S_e$ of the tangent space $T$ to $\NL(\gamma)$ at $Y$  is contained in $I_e$.
\end{lemma}

\begin{proof}
Fix a tangent vector $\xi \in T$. Then $\xi$ can be lifted to an element $G$ of $S_e$. Similarly $\gamma$ can be lifted to an element $H$ of $S_{(k+1)e-2k-2}$. Since $\xi$ is a tangent vector to the Hodge locus of $\gamma$ we find that 
 for every  $t\neq k$ the $(t,2k-t)$-part of $\xi \cup \gamma$ is zero. In particular, taking $t=k$ and using the identification with the Jacobian ring, we find that the product $GH$ equals $0$ in $(S/J)_{(k+2)e-2k-2}$. This implies that $(GH) S_{ke-2k-2}=0$ in $S/J$ and hence $GS_{ke-2k-2}\subset I_{(k+1)e-2k-2}$. Recall that $S/I$ is Artinian Gorenstein of socle degree $(k+1)e-2k-2$. Hence the existence of the non-degenerate pairing forces the class of $G$ to be zero in $S/I$, i.e, $G\in I_e$. 
 \end{proof}

\begin{lemma} We have that
\[ I =(P_1,\dots,P_{k+1},Q_1,\dots,Q_{k+1})\]
and $T=I_e$.
\end{lemma}

\begin{proof}
Since $Z$ is contained in $Y$ and $\gamma=[Z]_{\prim}$, it is obvious that $I(Z)_e$ is contained in the Zariski closure of the analytic scheme $\NL(\gamma)$. Since $I(Z)_e$ is a vector space  it is also contained in the tangent space at $Y$. Hence $I(Z)_e \subset T\subset I_e$. From the Gorenstein property and the fact that $I(Z)$ is generated in degree less then $e$  it follows that $I(Z)\subset I$.

Recall that $Z$ is given by  $P_1=\dots=P_{k+1}=0$ and $F=\sum P_iQ_i$. Suppose  that $Z'$ is obtained from $Z$ by replacing one of the $P_i$ with  $Q_i$ then the corresponding classes in $H^{2k}(Y,\mathbf{Q})_{\prim}$ differ by a sign. The ideal $I$ depends only on the subspace generated by $\gamma$, hence the ideal associated with $Z$ and with $Z'$ coincide. In particular $Q_i \in I$ for all $i$,
\[ I'=(P_1,\dots,P_{k+1},Q_1,\dots,Q_{k+1}) \subset I\]
and $I'_e\subset T$.
Moreover, as remarked above, we have that $S/I'$ is Artinian Gorenstein of socle degree $(k+1)e-2k-2$. 
Now by construction $S/I$ is Artinian Gorenstein of socle degree $(k+1)e-2k-2$. 

Hence $S/I$ is a quotient of $S/I'$ and both are Artinian Gorenstein rings with the same socle degree. It follows that $I=I'$ by Remark~\ref{rmkQuo}.
In particular,  $I'_e\subset T\subset I_e=I'_e$.
\end{proof}

Combining the previous lemmata we obtain
\begin{proposition} \label{propHLdim} Suppose $Y\subset \mathbf{P}^{2k+1}$ is a smooth hypersurface of degree $e$, containing a complete intersection $Z$ of multidegree $(d_1,\dots,d_{k+1})$, with $1\leq d_i\leq e-1$ for $i=1,\dots,k+1$. Let $I$ be any complete intersection ideal of multidegree $(d_1,\dots,d_{k+1},e-d_{k+1},\dots,e-d_1)$. Then the codimension in $S_e$ of the tangent space to $\NL([Z]_{\prim})$ at $Y$ equals $h_I(e)$.
\end{proposition}

We can now prove our main result:

\begin{theorem}
Suppose $Y\subset \mathbf{P}^{2k+1}$ is a smooth hypersurface of degree $e$, containing a complete intersection $Z$ of multidegree $(d_1,\dots,d_{k+1})$, with $1\leq d_i\leq e-1$ for $i=1,\dots,k+1$. %Let $I$ be a complete intersection ideal of multidegree $(d_1,\dots,d_{k+1},e-d_{k+1},\dots,e-d_1)$.
 Then  the irreducible component of the Hodge locus containing $\NL([Z]_{\prim})$ coincides with the locus $L(d_1,\dots,d_{k+1};e)$ of smooth hypersurfaces of degree $e$ containing a complete intersection of multidegree $(d_1,\dots,d_{k+1})$.% and both loci are smooth.
 
 Moreover, $\NL([Z]_{\prim})$ is smooth at $Y$.
\end{theorem}

\begin{proof}
Without loss of generality we may assume that each of the $d_i$ is at most $e/2$. 

Let $\gamma=[Z]_{\prim}$.
The locus $L(d_1,\dots,d_{k+1};e)$ is clearly irreducible. 
We have now the following series of inequalities:
\[ \dim_Y \NL(\gamma) \leq  \dim T_Y \NL(\gamma) = \dim L(d_1,\dots,d_k;e)\leq \dim_Y \NL(\gamma).\]
The first inequality is obvious, the equality follows from Proposition~\ref{propHSdim} and Proposition~\ref{propHLdim}. Let $U$ be a small neighborhood (in the analytic topology) of $[X]$ then one of the irreducible components of $L(d_1,\dots,d_{k+1};e)\cap U$ is contained in $\NL(\gamma)\cap U$. This yields the third inequality.% follows from the obvious inclusion 
%\[L(d_1,\dots,d_{k+1};e)\subset \NL(\gamma).\]

%Hence we find that $\NL(\gamma)$ is smooth. 
Since $\NL(\gamma)$  is connected it is then also irreducible. Hence $L(d_1,\dots,d_{k+1};e)$ and the Zariski closure $\overline{\NL(\gamma)}$ are irreducible varieties of the same dimension and the former is a closed subset of the latter, hence they coincide. 
\end{proof}

\begin{remark} The locus $\NL(\gamma)$ is only defined in a small analytic neighborhood of $Y$ and is smooth. However if $Y$ contains two or more complete intersections of the same multidegree, yielding different classes in $H^n(Y,\mathbf{Z})_{\prim}$ then the locus $L(d_1,\dots,d_{k+1},e)$ can be singular at $Y$.
\end{remark}

\bibliographystyle{plain}
\bibliography{remke2}
\end{document}